\documentclass[12pt, reqno]{amsart}

\usepackage{amsfonts}
\usepackage{amsmath}
\usepackage{amssymb}
\usepackage{amscd}
\usepackage[mathscr]{eucal}

\allowdisplaybreaks[1]

\renewcommand*{\bar}{\overline}

\newcommand{\TwoByTwo}[4]{\genfrac{(}{)}{0pt}{1}{\: #1\hspace{5pt} #2 \: }{\: #3 \hspace{5pt} #4\: }}
\newcommand{\TwoByTwoT}[4]{\genfrac{[}{]}{0pt}{1}{\, #1\hspace{3pt} #2 \, }{\, #3 \hspace{3pt} #4\, }}
\newcommand{\TwoByTwoJ}[4]{\genfrac{(}{)}{0pt}{1}{\: #1\hspace{0pt} #2 \: }{\: #3 \hspace{5pt} #4\: }}

\DeclareMathOperator{\ord}{ord}
\newcommand{\new}{\textnormal{new}}
\newcommand{\rT}{\mathrm{T}}

\makeatletter
\def\imod#1{\allowbreak\mkern2mu({\operator@font mod}\,\,#1)}
\makeatother

\makeatletter
\def\jmod#1{\allowbreak\mkern5mu({\operator@font mod}\,\,#1)}
\makeatother

\theoremstyle{plain}

\newtheorem{theorem}[equation]{Theorem}
\newtheorem{lemma}[equation]{Lemma}
\newtheorem{prop}[equation]{Proposition}
\newtheorem{cor}[equation]{Corollary}
\theoremstyle{definition}

\newtheorem{remark}[equation]{Remark}

\newtheorem*{ackn}{Acknowledgements}

\numberwithin{equation}{section}

\begin{document}

\title[A hypergeometric function associated to a Siegel eigenform]{A finite field hypergeometric function associated to eigenvalues of a Siegel eigenform}

\author{Dermot M\lowercase{c}Carthy}
\address{Department of Mathematics \& Statistics\\
Texas Tech University\\
Lubbock, TX 79409-1042\\
USA}
\email{dermot.mccarthy@ttu.edu}

\author{Matthew A. Papanikolas}
\address{Department of Mathematics \\
Texas A\&M University \\
College Station, TX 77843-3368\\
USA}
\email{map@math.tamu.edu}

\thanks{This project was partially supported by NSF Grant DMS-1200577.}

\date{November 4, 2014}

\begin{abstract}
Although links between values of finite field hypergeometric functions and eigenvalues of elliptic modular forms are well known, we establish in this paper that there are also connections to eigenvalues of Siegel modular forms of higher degree.
Specifically, we relate the eigenvalue of the Hecke operator of index $p$ of a Siegel eigenform of degree $2$ and level $8$ to a special value of a ${}_4F_3$-hypergeometric function.
\end{abstract}

\maketitle


\section{Introduction and Statement of Results}
One of the more interesting applications of hypergeometric functions over finite fields is their links to elliptic modular forms and in particular Hecke eigenforms \cite{A, AO, E, FOP, F2, FMcC, L, McC5, M, O, P}.
It is anticipated that these links represent a deeper connection that also encompasses Siegel modular forms of higher degree, and the purpose of this paper is to provide new evidence in this direction.
Specifically, we relate the eigenvalue of the Hecke operator of index $p$ of a certain Siegel eigenform of degree $2$ and level $8$ to a special value of a ${}_4F_3$-hypergeometric function over $\mathbb{F}_p$.
We believe this is the first result connecting hypergeometric functions over finite fields to Siegel modular forms of degree $>1$.

Hypergeometric functions over finite fields were originally defined by Greene \cite{G}, who first established these functions as analogues of classical hypergeometric functions.
Functions of this type were also introduced by Katz \cite{K} about the same time.
In the present article we use a normalized version of these functions defined by the first author \cite{McC6}, which is more suitable for our purposes.
The reader is directed to \cite[\S 2]{McC6} for the precise connections among these three classes of functions.

Let $\mathbb{F}_{p}$ denote the finite field with $p$ elements, $p$ a prime, and let $\widehat{\mathbb{F}^{*}_{p}}$ denote the group of multiplicative characters of $\mathbb{F}^{*}_{p}$.
We extend the domain of $\chi \in \widehat{\mathbb{F}^{*}_{p}}$ to $\mathbb{F}_{p}$ by defining $\chi(0):=0$ (including for the trivial character $\varepsilon$) and denote $\bar{\chi}$ as the inverse of $\chi$.
Let $\theta$ be a fixed non-trivial additive character of $\mathbb{F}_p$ and for $\chi \in \widehat{\mathbb{F}^{*}_{p}}$ we have the Gauss sum $g(\chi):= \sum_{x \in \mathbb{F}_p} \chi(x) \theta(x)$.
Then for $A_0,A_1,\dotsc, A_n, B_1, \dotsc, B_n \in \widehat{\mathbb{F}_p^{*}}$ and $x \in \mathbb{F}_p$ define
\begin{multline}\label{def_HypFnFF}
{_{n+1}F_{n}} {\biggl( \begin{array}{cccc} A_0, & A_1, & \dotsc, & A_n \\
 \phantom{B_0,} & B_1, & \dotsc, & B_n \end{array}
\Big| \; x \biggr)}_{p}\\
:= \frac{1}{p-1}  \sum_{\chi \in \widehat{\mathbb{F}_p^{*}}}
\prod_{i=0}^{n} \frac{g(A_i \chi)}{g(A_i)}
\prod_{j=1}^{n} \frac{g(\bar{B_j \chi})}{g(\bar{B_j})}
 g(\bar{\chi})
 \chi(-1)^{n+1}
 \chi(x).
\end{multline}

One of the first connections between finite field hypergeometric functions and the coefficients of elliptic modular forms is due to Ahlgren and Ono~\cite{AO}, who proved the following.  Throughout we let $\phi  \in \widehat{\mathbb{F}_p^{*}}$ denote the Legendre symbol $\bigl(\frac{\cdot}{p}\bigr)$.

\begin{theorem}[{Ahlgren, Ono~\cite[Thm.~6]{AO}}]\label{thm_AO}
Consider the unique newform $g \in S_4(\Gamma_0(8))$ and the integers $d(n)$ defined by
\[
  g(z) = \sum_{n=1}^\infty d(n) q^n = q \prod_{m=1}^\infty (1-q^{2m})^4(1-q^{4m})^4,
  \quad q := e^{2\pi i z}.
\]
Then for an an odd prime $p$,
\[
  {_{4}F_3} \biggl( \begin{array}{cccc} \phi, & \phi, & \phi, & \phi \vspace{.02in}\\
\phantom{\phi} & \varepsilon, & \varepsilon, & \varepsilon \end{array}
\Big| \; 1 \biggr)_{p}
=d(p) + p.
\]
\end{theorem}

\noindent Another result relating ${_{n+1}F_{n}}$ to the Fourier coefficients of an elliptic modular form is due to  Mortenson \cite{M}.

\begin{theorem}[{Mortenson \cite[Prop.~4.2]{M}}] \label{thm_3F2_16}
Consider the unique newform $h \in S_3(\Gamma_0(16), (\tfrac{-4}{\cdot}))$ and the integers $c(n)$ defined by
\[
h(z) = \sum_{n=1}^{\infty}c(n)q^n = q\prod_{m=1}^{\infty}(1-q^{4m})^6.
\]
Then for $p$ an odd prime,
\begin{equation*}
{_{3}F_2} \biggl( \begin{array}{ccc} \phi, & \phi, & \phi \vspace{.02in}\\
\phantom{\phi} & \varepsilon, & \varepsilon \end{array}
\Big| \; 1 \biggr)_{p}
=c(p).
\end{equation*}
\end{theorem}
\noindent(Note that the original statements of Theorems \ref{thm_AO} and \ref{thm_3F2_16} were expressed in terms of Greene's hypergeometric function. We  have reformulated them in terms of ${_{n+1}F_{n}}$.)

Our main result below (Theorem~\ref{thm_4F3}) concerns the evaluation ${}_4F_3( \phi, \phi, \phi, \phi; \varepsilon, \varepsilon, \varepsilon | {-1})_p$, which we relate to eigenvalues of a Siegel eigenform of degree $2$.
Before stating this result  we recall some fundamental facts about Siegel modular forms (see \cite{Ev, Ev2, Kl} for further details).

Let $\mathbb{A}^{m\times n}$ denote the set of all $m \times n$ matrices with entries in the set $\mathbb{A}$.  For a matrix $M$ we let ${^tM}$ denote its transpose, and if $M$ has entries in $\mathbb{C}$, we let $\textup{Tr}(M)$ denote its trace and $\textup{Im}(M)$ its imaginary part.
We denote the $r \times r$ identity matrix by $I_r$.  If a matrix $M \in \mathbb{R}^{n \times n}$ is positive definite, then we write $M > 0$, and if $M$ is positive semi-definite, we write $M \geq 0$.
The Siegel half-plane $\mathbb{H}^2$ of degree $2$ is defined by
$$\mathbb{H}^2:=  \left\{ Z \in \mathbb{C}^{2 \times 2} \mid {^tZ} = Z, \textup{Im}(Z)>0\right\}.$$
Let
$$\Gamma^2:=\textup{Sp}_{4}(\mathbb{Z})= \left\{ M \in \mathbb{Z}^{4 \times 4} \mid {^tM}JM = J \right\} ,\qquad J=\TwoByTwoJ{0}{\phantom{-1} I_2}{-I_2}{0},$$
be the Siegel modular group of degree $2$ and let
$$\Gamma^2(q) := \left\{M \in \Gamma^2 \mid M \equiv I_{4} \jmod{q} \right\}$$
be its principal congruence subgroup of level $q\in \mathbb{Z}^{+}$. If $\Gamma^{\prime}$ is a subgroup of $\Gamma^2$ such that $\Gamma^2(q) \subset \Gamma^{\prime}$ for some minimal $q$, then we say $\Gamma^{\prime}$ is a congruence subgroup of degree $2$ and level~$q$.
The modular group $\Gamma^2$ acts on $\mathbb{H}^2$ via the operation
$$M \cdot Z = \left(AZ+B\right) \left(CZ+D\right)^{-1}$$
where $M = \TwoByTwo{A}{B}{C}{D} \in \Gamma^2$, $Z \in \mathbb{H}^2$.
Let $\Gamma^{\prime}$ be a congruence subgroup of degree $2$ and level~$q$. A holomorphic function $F:\mathbb{H}^2 \to \mathbb{C}$ is called a Siegel modular form of degree $2$, weight $k\in \mathbb{Z}^{+}$ and level~$q$ on $\Gamma^{\prime}$ if
\[
F|_kM(Z):=\textup{det}(CZ+D)^{-k} \, F(M \cdot Z)   =  F(Z)
\]
for all $M = \TwoByTwo{A}{B}{C}{D} \in \Gamma^{\prime}$.  We note that the desired boundedness of $F|_k M(Z)$, for any $M \in \Gamma^2$, when $\textup{Im}(Z) - cI_2 \geq 0$, with fixed $c >0$, is automatically satisfied by the Koecher principle.
The set of all such modular forms is a finite dimensional vector space over $\mathbb{C}$, which we denote $M_k^2(\Gamma^{\prime})$.
Every $F \in M_k^2(\Gamma^{\prime})$ has a Fourier expansion of the form
$$F(Z) = \sum_{N \in \mathcal{R}^2} a(N) \exp \left(\tfrac{2 \pi i}{q} \, \textup{Tr}(NZ) \right)$$
where $Z \in \mathbb{H}^2$ and
$$\mathcal{R}^2 = \left\{N=(N_{ij}) \in \mathbb{Q}^{2 \times 2} \mid {^tN} = N \geq 0, N_{ii}, 2 N_{ij} \in \mathbb{Z} \right\}.$$
We call $F \in M_k^2(\Gamma^{\prime})$ a cusp form if $a(N)=0$ for all $N \not> 0$ and denote the space of such forms $S_k^2(\Gamma^{\prime})$.

The Igusa theta constant of degree 2 with characteristic
$m=(m^{\prime},m^{\prime \prime}) \in \mathbb{C}^{1 \times 4}$, $m^{\prime},m^{\prime \prime} \in \mathbb{C}^{1 \times 2}$ is defined by
$$\Theta_m(Z) = \sum_{n \in \mathbb{Z}^{1 \times 2}} \exp \left( \pi i \left\{(n+m^{\prime}) Z \: {^t(n+m^{\prime})} + 2 (n+m^{\prime}) {\:^t m^{\prime \prime}} \right\} \right).$$
If $m=\frac{1}{2} (a,b,c,d)$ then we will write $\Theta \TwoByTwoT{a}{b}{c}{d}$ for $\Theta_m$.

In \cite{vGvS}, van Geemen and van Straten exhibited several Siegel cusp forms of degree $2$ and level $8$ which are products of theta constants.
The principal form of interest to us is
\begin{equation}
F_7(Z):=
\Theta \TwoByTwoT{0}{0}{0}{0}(Z) \cdot \Theta \TwoByTwoT{0}{0}{0}{0}(Z) \cdot \Theta \TwoByTwoT{1}{0}{0}{0}(Z) \cdot \Theta \TwoByTwoT{0}{1}{0}{0}(Z) \cdot \Theta \TwoByTwoT{0}{0}{0}{1}(Z) \cdot \Theta \TwoByTwoT{0}{0}{1}{1}(Z),
\end{equation}
which lies in $S_3^2(\Gamma^2(4,8))$, where
$$\Gamma^2(8) \subset \Gamma^2(4,8):=\left\{ \TwoByTwo{A}{B}{C}{D} \in \Gamma^2(4) \mid \textup{diag}(B) \equiv \textup{diag}(C) \equiv 0 \jmod 8 \right\} \subset \Gamma^2(4).$$
Furthermore, $F_7(Z)$ is an eigenform \cite{vGvS} in the sense that it is a simultaneous eigenfunction for the all Hecke operators acting on $M_3^2(\Gamma^{2}(8))$. (See \cite{Ev} for details on Hecke operators for Siegel forms of degree 2 with level.)
In particular, for $p$ an odd prime we can define the eigenvalue $\lambda(p) \in \mathbb{C}$ of $F_7$ by
$$T(p) \, F_7 = \lambda(p) \, F_7,$$
where $T(p)$ is the Hecke operator of index $p$.

Our main result (Theorem~\ref{thm_4F3}) relies on a connection between the Andrianov $L$-function $L^a(s,F_7)$ of $F_7$ and the tensor product $L$-function $L(s,f_1\otimes f_2)$ of two elliptic newforms, $f_1 \in S_2(\Gamma_0(32))$ and $f_2 \in S_3(\Gamma_0(32), (\tfrac{-4}{\cdot}))$, where
\[
f_1(z) = \sum_{n=1}^\infty a(n) q^n = q \prod_{m=1}^\infty (1-q^{4n})^2 (1-q^{8n})^2
\]
and taking $i = \sqrt{-1}$,
\[
f_2(z) = \sum_{n=1}^{\infty}b(n)q^n
= q + 4iq^3 +2 q^5 -8iq^7 -7q^9 - 4iq^{11} -14 q^{13} + 8i q^{15}+18 q^{17}+ \cdots.
\]
Van Geemen and van Straten~\cite[\S 8.7]{vGvS} conjectured that $L^a(s,F_7)$ and $L(s,f_1 \otimes f_2)$ have essentially the same Euler factors up to a twist of $F_7$ (see the beginning of the proof of Theorem~\ref{thm_4F3} for a precise statement). Our understanding is that this conjecture has been resolved recently by Okazaki~\cite{Ok} (see also ~\cite[\S 1]{Ok2} for additional discussion on this problem).  As a consequence, for any odd prime $p$, we find (see Section~\ref{sec_proofs} for details) that
\begin{equation} \label{lambdap}
  \lambda(p) = \begin{cases}
  a(p)b(p) & \textnormal{if $p \equiv 1 \pmod{8}$,} \\
  -a(p)b(p) & \textnormal{if $p \equiv 5 \pmod{8}$,} \\
  0 & \textnormal{if $p \equiv 3 \pmod{4}$.}
  \end{cases}
\end{equation}

The main result of this paper relates the eigenvalues $\lambda(p)$ to a special value of a ${}_4F_3$-hypergeometric function over $\mathbb{F}_p$, conditional on the conjecture of van Geemen and van Straten.

\begin{theorem}\label{thm_4F3}
Fix the Dirichlet character $\xi(\cdot) = \bigl( \tfrac{2}{\cdot} \bigr)$ modulo $8$.  Assuming the conjecture of van Geemen and van Straten for $F_7$ \cite[\S 8.7]{vGvS}, for any odd prime $p$,
\begin{equation*}
{_{4}F_3} \biggl( \begin{array}{cccc} \phi, & \phi, & \phi, & \phi \vspace{.02in}\\
\phantom{\phi} & \varepsilon, & \varepsilon, & \varepsilon \end{array}
\Big| \; {-1} \biggr)_{p} = \xi(p)\lambda(p).
\end{equation*}
Moreover, when $p\equiv 3 \pmod 4$, $\lambda(p)=0$.
\end{theorem}

Now the Fourier coefficients $a(p)$ and $b(p)$ can themselves be related to hypergeometric functions over $\mathbb{F}_p$.  The first such connection is due to Ono, and the second is established in the present paper.

\begin{theorem}[{Ono~\cite[Thm.~2]{O}}] \label{thm_2F1_32}
For $p$ an odd prime,
\[
{_{2}F_1} \biggl( \begin{array}{ccc} \phi, & \phi \vspace{.02in}\\
\phantom{\chi_4} & \varepsilon \end{array}
\Big| \; {-1} \biggr)_{p}
=a(p).
\]
\end{theorem}

\begin{theorem}\label{thm_3F2_32}
If $p \equiv 1 \pmod 4$ is prime and $\chi_4 \in \widehat{\mathbb{F}_p^*}$ has order $4$, then
\begin{equation*}
{_{3}F_2} \biggl( \begin{array}{ccc} \chi_4, & \phi, & \phi \vspace{.02in}\\
\phantom{\chi_4} & \varepsilon, & \varepsilon \end{array}
\Big| \; 1 \biggr)_{p}
=b(p).
\end{equation*}
\end{theorem}

The proof of Theorem~\ref{thm_3F2_32} follows lines of inquiry similar to those of Ahlgren and Ono \cite{AO2}, Ahlgren \cite{A2}, and Frechette, Ono, and the second author \cite{FOP}.
We use the Eichler-Selberg trace formula for Hecke operators to isolate the Fourier coefficients of the form and connect these traces to hypergeometric values by counting isomorphism classes of members of the Legendre family of elliptic curves with prescribed torsion.

The final step in proving Theorem~\ref{thm_4F3} is to appeal to a result of the first author \cite[Thm.~1.5]{McC6} (see Theorem~\ref{thm_4F3_3F2}), which provides a finite field version of a well-poised ${}_4F_3$-hypergeometric identity of Whipple.  From this we deduce that ${}_4F_3( \phi, \phi, \phi, \phi; \varepsilon, \varepsilon, \varepsilon | {-1})_p = 0$ for $p \equiv 3\pmod{4}$ and that for $p \equiv 1\pmod{4}$,
\begin{equation}\label{for_4F3_3F22F1}
  {_{4}F_3} \biggl( \begin{array}{cccc} \phi, & \phi, & \phi, & \phi \vspace{.02in}\\
\phantom{\phi} & \varepsilon, & \varepsilon, & \varepsilon \end{array}
\Big| \; {-1} \biggr)_{p}
= {_{2}F_1} \biggl( \begin{array}{cc} \phi, & \phi \vspace{.02in}\\
\phantom{\phi} & \varepsilon,  \end{array}
\Big| \; {-1} \biggr)_{p}
 \cdot
{_{3}F_2} \biggl( \begin{array}{ccc} \chi_4, & \phi, & \phi \vspace{.02in}\\
\phantom{\chi_4} & \varepsilon, & \varepsilon \end{array}
\Big| \; 1 \biggr)_{p}.
\end{equation}
Combining \eqref{lambdap}, Theorems \ref {thm_2F1_32} and \ref{thm_3F2_32}, and \eqref{for_4F3_3F22F1} yields the desired result.

The remainder of this paper is organized as follows.
In Section \ref{sec_classno} we recall some properties of class numbers of orders of imaginary quadratic fields and outline their relationship to isomorphism classes of elliptic curves over $\mathbb{F}_p$.
Section \ref{sec_trace} outlines our use of  the Eichler-Selberg trace formula. Properties of the trace of Frobenius of the Legendre family of elliptic curves are developed in Section \ref{sec_ec}.
The proofs of Theorems~\ref{thm_4F3} and~\ref{thm_3F2_32} are contained in Section \ref{sec_proofs}. Finally, we make some closing remarks in Section \ref{sec_remarks}.

\begin{ackn}
The authors are extremely grateful to F.~Rodriguez Villegas for pointing out to them the potential connection between the Siegel eigenforms and hypergeometric function values and in particular for alerting them to the example in \cite{vGvS} as a source for our identities.  The authors especially thank T.~Okazaki for his generosity in sharing his results and answering questions on his forthcoming paper.  The authors also thank R.~Osburn for helpful advice and suggestions.
\end{ackn}


\section{Class Numbers and Isomorphism Classes of Elliptic Curves}\label{sec_classno}

In this section we recall some properties of class numbers of orders of imaginary quadratic fields. In particular we will note their relationship to isomorphism classes of elliptic curves over $\mathbb{F}_p$.
(See \cite{Co, Sc} for further details.)

We first recall some notation.
For $D<0$, $D\equiv 0,1 \pmod 4$, let $\mathcal{O}(D)$ denote the unique  imaginary quadratic order of discriminant $D$.
Let $h(D)=h(\mathcal{O}(D))$ denote the class number of $\mathcal{O}(D)$, i.e., the order of the ideal class group of  $\mathcal{O}(D)$.
Let $\omega(D)=\omega(\mathcal{O}(D)):= \tfrac{1}{2} |\mathcal{O}(D)^{\ast}|$ where $\mathcal{O}(D)^{\ast}$ is the group of units of $\mathcal{O}(D)$.
For brevity, we let $h^{\ast}(D):=h(D)/ \omega(D).$ We also define
\begin{equation}\label{def_BigH}
H(D):= \sum_{\mathcal{O}(D) \subseteq \mathcal{O}^{\prime} \subseteq \mathcal{O}_{\max}} h(\mathcal{O}^{\prime})  \quad \textup{and } \quad H^{\ast}(D):= \sum_{\mathcal{O}(D) \subseteq \mathcal{O}^{\prime} \subseteq \mathcal{O}_{\max}} h^{\ast}(\mathcal{\mathcal{O}^{\prime}})
\end{equation}
where the sums are over all orders $\mathcal{O}^{\prime}$ between $\mathcal{O}(D)$ and the maximal order $\mathcal{O}_{\max}$.
We note that $H^{\ast}(D) = H(D)$ unless $\mathcal{O}_{\max}=\mathbb{Z}[\sqrt{-1}]$ or $\mathbb{Z}\left[\tfrac{-1+\sqrt{-3}}{2}\right]$. In these exceptional cases, $H(D)$ is greater by $\frac{1}{2}$ and $\frac{2}{3}$ respectively, as only the term corresponding to $\mathcal{O}_{\max}$ in each sum differs.

If $\mathcal{O}$ has discriminant $D$ and $\mathcal{O}^{\prime} \subseteq \mathcal{O}$ is an order such that $[\mathcal{O}:\mathcal{O}^{\prime}]=f$, then the discriminant of $\mathcal{O}^{\prime}$ is $f^2 D$. We will need the following lemma which relates class numbers of certain orders.

\begin{lemma}[{\cite[Cor.~7.28]{Co}}] \label{lem_hrel}
Let $\mathcal{O}$ be an order of discriminant $D$ in an imaginary quadratic field, and let $\mathcal{O}^{\prime} \subseteq \mathcal{O}$ be an order with $[\mathcal{O}:\mathcal{O}^{\prime}]=f$. Then
$$ h^{\ast}(\mathcal{\mathcal{O}^{\prime}}) = h^{\ast}(\mathcal{\mathcal{O}})\cdot f \prod_{\substack{l \mid f \\ l \, prime}} \left(1-\left(\frac{D}{l}\right)\frac{1}{l}\right),$$
where $\left(\frac{D}{l}\right)$ is the Kronecker symbol.
\end{lemma}

Finally we present a result of Schoof which relates these class numbers to the number of isomorphism classes of elliptic curves over $\mathbb{F}_p$.

\begin{theorem}[{Schoof~\cite[(4.5)--(4.9)]{Sc}}] \label{thm_Schoof}
Let $p \geq 5$ be prime. Suppose $n \in \mathbb{Z}^{+}$, $s \in \mathbb{Z}$ satisfy $s^2  \leq 4p$, $p \nmid s$, $n \mid (p-1)$ and $n^2 \mid (p+1-s)$.
Then the number of isomorphism class of elliptic curves over $\mathbb{F}_p$  whose group of $\mathbb{F}_p$-rational points has order $p+1-s$ and contains $\mathbb{Z} / n \mathbb{Z} \times  \mathbb{Z} /  n \mathbb{Z}$ is $H\left(\tfrac{s^2-4p}{n^2}\right)$.
\end{theorem}


\section{Eichler-Selberg Trace Formula}\label{sec_trace}
Let $\textup{Tr}_k(\Gamma_0(N), \chi,p)$ denote the trace of the classical Hecke operator $\rT(p)$ acting on $S_k(\Gamma_0(N), \chi)$.
Similarly let $\textup{Tr}^{\new}_k(\Gamma_0(N), \chi,p)$ denote the trace of  $\rT(p)$ on the new subspace $S^{\new}_k(\Gamma_0(N), \chi)$.
The Eichler-Selberg trace formula gives a precise description of $\textup{Tr}_k(\Gamma_0(N), \chi,p)$.
In this section we use a version of the trace formula due to Hijikata \cite[Thm.~2.2]{HPS} to evaluate $\textup{Tr}_3(\Gamma_0(16), (\tfrac{-4}{\cdot}),p)$ and $\textup{Tr}_3(\Gamma_0(32), (\tfrac{-4}{\cdot}),p)$.

We can simplify the trace formula as given in \cite[Thm.~2.2]{HPS} using the particulars of our spaces and \cite[Lem.~2.5]{HPS}.
This is a long and tedious process but not particularly difficult so we omit the details here. However we note that we have used some elementary facts which we present in the following proposition.
For $p$ an odd prime and $s$ an integer such that $s^2<4p$ we may write $s^2-4p=t^2D$, where $D$ is a fundamental discriminant of some imaginary quadratic field. For a given $p$ and $s$,  let $a:=\ord_2(t)$.
\begin{prop}\label{prop_pDa}
If $s\equiv (p+1) \pmod 8$ then
\begin{itemize}
\item if $p \equiv 1 \pmod 8$ and $D$ is odd, then $a>2$;
\item if $p \equiv 5 \pmod 8$ and $D$ is odd, then $a=2$;
\item if $p \equiv 5 \pmod 8$ and $D$ is even, then $a<2$.
\end{itemize}
\end{prop}
\noindent We have also used \cite[Lem.~4.2]{AO2}. These simplifications yield the following two Theorems.

\begin{theorem}\label{thm_Trace16}
If $p$ is an odd prime, then
\begin{equation*}
\textup{Tr}_3(\Gamma_0(16), (\tfrac{-4}{\cdot}),p)
=\begin{cases}
 0 & \textup{if } p \equiv 3 \jmod 4,\\[6pt]
-6-\displaystyle\sum_{\substack{0< |s| <2 \sqrt{p} \\ s  \equiv p+1 \imod {16}}} s \; \displaystyle\sum_{\substack{f \mid t\\ s^2-4p=t^2D}} h^*\left(\tfrac{s^2-4p}{f^2}\right) c_1(s,f) & \textup{if } p \equiv 1 \jmod 4,
 \end{cases}
\end{equation*}
where $D$ is a fundamental discriminant and $c_1(s,f)$ is described in the table below with  $a:=\ord_2(t)$ and $b:=\ord_2(f)$.

\begin{center}
\begin{tabular}{|r||c|c|c|}
\hline
$c_1(s,f)$ &  $D$ even & $D\equiv 1 \jmod{8}$ & $D\equiv 5 \jmod{8}$\\
\hline
$a-b=0$ & $0$ & $2$ & $0$\\
$a-b=1$ & $0$ & $6$ & $0$\\
$a-b=2$ & $6$ & $8$ & $4$\\
$a-b\geq 3$ & $6$ & $6$ & $6$\\
\hline
\end{tabular}
\end{center}
\end{theorem}

\begin{theorem}\label{thm_Trace32}
If $p$ is an odd prime, then
\begin{equation*}
\textup{Tr}_3(\Gamma_0(32), (\tfrac{-4}{\cdot}),p)
=\begin{cases}
 0 & \textup{if } p \equiv 3 \jmod 4,\\[6pt]
-8-\displaystyle\sum_{\substack{0< |s| <2 \sqrt{p} \\ s  \equiv p+1 \imod {16}}} s \; \displaystyle\sum_{\substack{f \mid t\\ s^2-4p=t^2D}} h^*\left(\tfrac{s^2-4p}{f^2}\right) c_2(s,f) & \textup{if } p \equiv 1 \jmod 4,
 \end{cases}
\end{equation*}
where $D$ is a fundamental discriminant and $c_2(s,f)$ is described in the table below with  $a:=\ord_2(t)$ and $b:=\ord_2(f)$.

\begin{center}
\begin{tabular}{|r||c|c|c|}
\hline
$c_2(s,f)$ &  $D$ even & $D\equiv 1 \jmod{8}$ & $D\equiv 5 \jmod{8}$\\
\hline
$a-b=0$ & $0$ & $2$ & $0$\\
$a-b=1$ & $0$ & $6$ & $0$\\
$a-b=2$ & $4$ & $12$ & $0$\\
$a-b\geq 3$ & $8$ & $8$ & $8$\\
\hline
\end{tabular}
\end{center}
\end{theorem}

The dimension of $S_3(\Gamma_0(16), (\tfrac{-4}{\cdot}))$ is one, so we can combine Theorems \ref{thm_Trace16} and \ref{thm_Trace32} to evaluate $\textup{Tr}^{\new}_3(\Gamma_0(32), (\tfrac{-4}{\cdot}),p)$.

\begin{cor}\label{cor_Trace32New}
If $p$ is an odd prime, then
\begin{equation*}
\textup{Tr}^{\new}_3(\Gamma_0(32), (\tfrac{-4}{\cdot}),p)
=\begin{cases}
 0 & \textup{if } p \equiv 3 \jmod 4,\\[6pt]
4-\displaystyle\sum_{\substack{0< |s| <2 \sqrt{p} \\ s  \equiv p+1 \imod {16}}} s \; \displaystyle\sum_{\substack{f \mid t\\ s^2-4p=t^2D}} h^*\left(\tfrac{s^2-4p}{f^2}\right) c_3(s,f) & \textup{if } p \equiv 1 \jmod 4,
 \end{cases}
\end{equation*}
where $D$ is a fundamental discriminant and $c_3(s,f)$ is described in the table below with  $a:=\ord_2(t)$ and $b:=\ord_2(f)$.

\begin{center}
\begin{tabular}{|r||c|c|c|}
\hline
$c_3(s,f)$ &  $D$ even & $D\equiv 1 \jmod{8}$ & $D\equiv 5 \jmod{8}$\\
\hline
$a-b=0$ & $0$ & $-2$ & $0$\\
$a-b=1$ & $0$ & $-6$ & $0$\\
$a-b=2$ & $-8$ & $-4$ & $-8$\\
$a-b\geq 3$ & $-4$ & $-4$ & $-4$\\
\hline
\end{tabular}
\end{center}
 \end{cor}


\section{Legendre Family of Elliptic Curves}\label{sec_ec}
In this section we recall some properties of elliptic curves, and in particular properties of the Legendre family of elliptic curves over finite fields. For further details please refer to \cite{Kn, Si}.
We will also develop some preliminary results which we will use in Section \ref{sec_proofs}.

\begin{theorem}[{\cite[Thm.~4.2]{Kn}}]\label{thm_2P}
Suppose $E$ is an elliptic curve over a field $K$, $\textup{char}(K) \neq 2$, $3$, given by
$y^2=(x-a)(x-b)(x-c)$,
with $a$, $b$, $c \in K$ all distinct.
Then, given a point $P=(x,y) \in E(K)$ there exists $Q \in E(K)$ with $P=[2]Q$ if and only if $x-a$, $x-b$ and $x-c$ are all squares in $K$.
\end{theorem}

The Legendre family of elliptic curves $E_{\lambda}$ over $\mathbb{F}_p$ is given by
\begin{equation}\label{def_ELambda}
E_{\lambda}: y^2=x(x-1)(x-\lambda), \qquad \lambda \in \mathbb{F}_p \setminus \lbrace 0, 1 \rbrace \: .
\end{equation}

\begin{prop}\label{prop_Legendrej}
\begin{enumerate}
\item The $j$-invariant of $E_{\lambda}$ is
$$ j(E_{\lambda}) = 2^8 \frac{(\lambda^2 - \lambda +1)^3}{\lambda^2(\lambda-1)^2}.$$
\item $j(E_{\lambda}) = 1728$ if and only if $\lambda \in \{2, -1, \frac{1}{2}\}$.
\item $j(E_{\lambda}) = 0$ if and only if $\lambda^2-\lambda+1 = 0$, i.e., $\lambda \in \{\frac{1 \pm \sqrt{-3}}{2}\}$.
\item The map $\lambda \to j(E_{\lambda})$ is surjective and six-to-one except above $j=0$ and $j=1728$. In particular,
$$ \left\{\lambda, \frac{1}{\lambda}, 1-\lambda, \frac{1}{1-\lambda}, \frac{\lambda}{\lambda-1}, \frac{\lambda-1}{\lambda}\right\} \to j(E_{\lambda}).$$
\end{enumerate}
\end{prop}

If $t \in \mathbb{F}_p\setminus \{0\}$, then we define the $t$-quadratic twist of $E_{\lambda}$ by $E^t_\lambda: y^2 = x(x-t)(x-t\lambda).$
If $t$ is a square in $\mathbb{F}_p$ then $E_{\lambda}$ is isomorphic to $E^t_\lambda$ over $\mathbb{F}_p$.  Straightforward calculations yield the following result.

\begin{prop}\label{prop_Etwist}
Let $p\geq5$ be prime.
\textup{(1)} $E_{\lambda}$ is the $\lambda$ quadratic twist of  $E_{\frac{1}{\lambda}}$.
\textup{(2)} $E_{\lambda}$ is the $-1$ quadratic twist of $E_{1-\lambda}$.
\textup{(3)} $E_{\lambda}$ is the $1-\lambda$ quadratic twist of $E_{\frac{\lambda}{\lambda-1}}$.
\textup{(4)} $E_{\lambda}$ is the $\lambda-1$ quadratic twist of $E_{\frac{1}{1-\lambda}}$.
\textup{(5)} $E_{\lambda}$ is the $-\lambda$ quadratic twist of $E_{\frac{\lambda-1}{\lambda}}$.
\end{prop}

\noindent The curve $E_{\lambda}$ has $3$ points of order $2$ namely, $(0,0)$, $(1,0)$ and $(\lambda,0)$.
The following lemma gives us certain information on the $2$-power torsion of $E_\lambda$ which we will require in Section \ref{sec_proofs}.
Please see \cite[V.5] {Kn}, \cite[Lemma 3.2]{A2} and \cite[Prop. 3.3]{AO2} for similar arguments.

\begin{lemma}\label{lem_8tors}
Let $p\geq5$ be prime and let $E_{\lambda}$ be defined by \eqref{def_ELambda}.
\begin{enumerate}
\item $E_{\lambda}( \mathbb{F}_p)$ contains $\mathbb{Z} / 2 \mathbb{Z} \times  \mathbb{Z} /  8 \mathbb{Z}$ but not $\mathbb{Z} / 4 \mathbb{Z} \times  \mathbb{Z} /  4\mathbb{Z}$ if $-1$ is a square, $\lambda$ is a fourth power and $\lambda-1$ is not a square in $\mathbb{F}_p$.
\item Let $E / \mathbb{F}_p$ be an elliptic curve such that $E( \mathbb{F}_p)$ contains $\mathbb{Z} / 2 \mathbb{Z} \times  \mathbb{Z} /  8 \mathbb{Z}$ but not $\mathbb{Z} / 4 \mathbb{Z} \times  \mathbb{Z} /  4\mathbb{Z}$.
Then, if  $p \equiv 1 \pmod 4$, $E$ is isomorphic to $E_\lambda$ for some $\lambda \in \mathbb{F}_p \setminus \lbrace 0, 1 \rbrace$ with $\lambda$ a fourth power and $\lambda-1$ not a square in $\mathbb{F}_p$.
\end{enumerate}
\end{lemma}

\begin{proof}
(1) Using Theorem \ref{thm_2P} we see that $(0,0)$ is a double but $(1,0)$ and $(\lambda,0)$ are not. Therefore $E_{\lambda}( \mathbb{F}_p)$ cannot contain $\mathbb{Z} / 4 \mathbb{Z} \times  \mathbb{Z} /  4 \mathbb{Z}$.
Let $t \in \mathbb{F}_p$ be fixed such that $t^2= \lambda$ and let $i \in \mathbb{F}_p$ be fixed such that $i^2= -1$.
Using the duplication formula (\cite[\S III.2]{Si}) we see that if $-1$ and $\lambda$ are squares then $Q_1 = (t, i t (t-1))$, $Q_2 = (t, -i t (t-1))$, $Q_3= (-t,  i t (t+1))$, $Q_4= (-t, - i t (t+1))$,
are four points of order $4$ whose double is $(0,0)$. If $\lambda-1 = t^2-1$ is not a square then either $t-1$ or $t+1$ is a square but not both. Therefore, using Theorem~\ref{thm_2P} we see that either $Q_1$ and $Q_2$ are doubles or $Q_3$ and $Q_4$ are doubles, but not both pairs.
Therefore $E_{\lambda}( \mathbb{F}_p)$ contains $\mathbb{Z} / 2 \mathbb{Z} \times  \mathbb{Z} /  8 \mathbb{Z}$.

(2) As $E$ has full $2$-torsion we can assume it is of the form $y^2 = x(x-\alpha)(x-\beta)$ for some $\alpha$, $\beta \in \mathbb{F}_p$. Also, $E$ contains $\mathbb{Z} / 2 \mathbb{Z} \times  \mathbb{Z} /  4 \mathbb{Z}$ but not $\mathbb{Z} / 4 \mathbb{Z} \times  \mathbb{Z} /  4 \mathbb{Z}$ so one and only one of $(0,0)$, $(\alpha,0)$ and $(\beta,0)$ is a double.
We can assume without loss of generality that $(0,0)$ is the double point. Therefore, if  $p \equiv 1 \pmod 4$, we see from Theorem \ref{thm_2P} that $\alpha$ and $\beta$ are squares but $\alpha - \beta$ is not. Letting $m^2= \alpha$ and $n^2 = \beta$ then $E$ has the form
$y^2=x(x-m^2)(x-n^2)$ which is isomorphic to $E_{\lambda}$ with $\lambda= \frac{n^2}{m^2}$. (Make the change of variables $x \to m^2 x$ and $y \to m^3y$ .) Therefore $\lambda$ is a square but $\lambda-1$ is not.
We see from part (1) that as $-1$ and $\lambda$ are both squares there are four points of order $4$ whose double is $(0,0)$. Two of of these points have $x$ coordinate of  $t$ and the other two have $x$ coordinate of $-t$, where $t^2= \lambda$.
As $E$ contains $\mathbb{Z} / 2 \mathbb{Z} \times  \mathbb{Z} /  8 \mathbb{Z}$ exactly two of these points of order 4 are doubles. Therefore, by Theorem \ref{thm_2P}, $t$ must be a square and either $t-1$ or $t+1$ is a square but not both.
As $t^2= \lambda$ and $(t-1)(t+1)=\lambda-1$ we see that $\lambda$ is a fourth power and $\lambda-1$ is not a square.
\end{proof}

For an elliptic curve $E / \mathbb{F}_p$ we define the integer $a_p(E)$ by
$a_p(E) := p +1 -  |E(\mathbb{F}_p)|$.  A curve $E / \mathbb{F}_p$ is supersingular if and only if $p \mid a_p(E)$, which if $p\geq 5$, is equivalent to $a_p(E)=0$ by the Hasse bound.
As usual if $E$ is given by $y^2 = f(x)$ then
\begin{equation}\label{for_a}
a_p(E) = - \sum_{x \in \mathbb{F}_p} \phi_p(f(x)),
\end{equation}
where $\phi_p(\cdot)$ is the Legendre symbol modulo $p$ (and the character of order $2$ of $\mathbb{F}_p^{*}$). We will often omit the subscript $p$ when it is clear from the context.
Consequently,
\begin{equation}\label{for_aTwist}
a_p(E_{\lambda}) = \phi(t) \, a_p(E^{t}_{\lambda}).
\end{equation}
We now prove some identities for $a_p(\lambda):=a_p(E_{\lambda})$.

\begin{lemma}\label{lem_ap3}
For $p$ an odd prime,
$\displaystyle \sum_{\substack{\lambda =2\\ \phi(\lambda)=1}}^{p-1} a_p(\lambda)
= -(1 + \phi(-1)).$
\end{lemma}

\begin{proof}
Using \eqref{for_a} we see that
\[
\sum_{\substack{\lambda =2\\ \phi(\lambda)=1}}^{p-1} a_p(\lambda)
=  \frac{1}{2}\sum_{t =2}^{p-2} a_p(t^2)
= - \frac{\phi(-1)}{2} \sum_{x\in \mathbb{F}_p \setminus \{0,1\}} \phi(x) \phi(x-1)  \sum_{t =2}^{p-2} \phi(t^2-x).
\]
Using properties of Jacobsthal sums (see \cite[\S 6.1]{BEW}) we note that, for $a \in \mathbb{F}_p \setminus \{0\}$, we have
$\sum_{t =1}^{p-1} \phi(t^2+a) = -(1+\phi(a))$ and $\sum_{x=1}^{p-1} \phi(x^2-x) = -1$.
Therefore, we have
$\sum_{t =2}^{p-2} \phi(t^2-x) = -(1+\phi(-x)+2 \phi(1-x))$, and
\begin{align*}
\sum_{\substack{\lambda =2\\ \phi(\lambda)=1}}^{p-1} a_p(\lambda)
&= \frac{\phi(-1)}{2} \sum_{x\in \mathbb{F}_p \setminus \{0,1\}} \phi(x^2-x)
+ \frac{1}{2} \sum_{x\in \mathbb{F}_p \setminus \{0,1\}} \phi(x-1)
+ \sum_{x\in \mathbb{F}_p \setminus \{0,1\}} \phi(x).
\end{align*}
We then apply the second Jacobsthal identity, noting that $\sum_{x \in \mathbb{F}_p} \phi(x) = 0$.
\end{proof}

\begin{lemma}\label{lem_ap4}
If $p \equiv 1 \pmod 4$, let $\chi_4 \in \widehat{\mathbb{F}_p^*}$ denote a character of order $4$.  Then
\begin{enumerate}
\item $\displaystyle \sum_{\substack{\lambda =2\\ \phi(\lambda(\lambda-1))=-1}}^{p-1} a_p(\lambda) \chi_4(\lambda(\lambda-1)) \phi(\lambda-1) = 0;$
\item $\displaystyle \sum_{\substack{\lambda =2\\ \phi(\lambda(\lambda-1))=1}}^{p-1} a_p(\lambda) \chi_4(\lambda(\lambda-1)) \phi(\lambda-1)
=  \sum_{\substack{\lambda =2\\ \phi(\lambda)=1}}^{p-1} a_p(\lambda) \chi_4(\lambda) \phi(\lambda-1).$
\end{enumerate}
\end{lemma}

\begin{proof}
The proofs of (1) and (2) are relatively routine applications of Proposition~\ref{prop_Etwist} together with \eqref{for_aTwist}, using the fact that $-1$ is a square in $\mathbb{F}_p$ if $p \equiv 1 \pmod{4}$.
We omit the details for reasons of brevity.
\end{proof}

\begin{lemma}\label{lem_ap6}
If $p \equiv 1 \pmod 4$, let $\chi_4 \in \widehat{\mathbb{F}_p^*}$ denote a character of order $4$.  Then
\begin{equation*}
\sum_{\substack{\lambda =2\\ \chi_4(\lambda) = 1 \\ \phi(\lambda-1)=-1}}^{p-1} a_p(\lambda)
=  \sum_{\substack{\lambda =2\\ \phi(\lambda)=1, \chi_4(\lambda)=-1 \\ \phi(\lambda-1) = 1}}^{p-1} a_p(\lambda) .
\end{equation*}
\end{lemma}

\begin{proof}
The curve $E_\lambda$ is 2-isogenous to the elliptic curve
$W_{\lambda}: y^2=x^3 + 2(1+\lambda)x^2 +(1-\lambda)^2 x$, $\lambda \in \mathbb{F}_p \setminus \lbrace 0, 1 \rbrace$,
via the maps $f_\lambda : W_{\lambda} \to E_\lambda$ and its dual $\tilde{f}_\lambda:  E_\lambda \to W_{\lambda}$ given by
\[
f_\lambda(x,y) = \left(\tfrac{y^2}{4x^2}, \tfrac{y((1-\lambda)^2-x^2)}{8x^2} \right),
\quad
\tilde{f}_\lambda(x,y) = \left(\tfrac{y^2}{x^2}, \tfrac{y(\lambda-x^2)}{x^2} \right),
\]
(see \cite[Ex.~III.4.5]{Si}).  In turn, when $\phi(\lambda)=1$ and $t \in \mathbb{F}_p$ is fixed such that $t^2= \lambda$, $W_{\lambda}$ is isomorphic over $\mathbb{F}_p$ to $E_{\psi}$, where $\psi= \bigl(\tfrac{1-t}{1+t}\bigr)^2$.
This follows from the change of variables $(x,y) \to (u^2 x, u^3 y)$ where $u^2 = -(1+t)^2.$
Noting that isogenous curves over $\mathbb{F}_p$ have the same number of $\mathbb{F}_p$-rational points,  we see that $a_p(\lambda) =  a_p(\psi)$.
Also, as $p \equiv 1 \pmod 4$, it is easy to check that  if $\chi_4(\lambda)=1$ and $\phi(\lambda-1)=-1$ then $\phi(\psi)=1$, $\chi_4(\psi)=-1$ and $\phi(\psi-1)=1$.
We now consider the sets
\begin{align*}
S_\lambda &= \left\{2 \leq \lambda \leq p-1 \mid \chi_4(\lambda)=1, \phi(\lambda-1)=-1 \right\},\\
S_\psi &= \left\{2 \leq \psi \leq p-1 \mid \phi(\psi)=1, \chi_4(\psi)=-1, \phi(\psi-1)=1 \right\}.
\end{align*}
If $\lambda \in S_\lambda$ then $\frac{1}{\lambda} \in S_\lambda$ also, and, as  $p \equiv 1 \pmod 4$,  $\lambda$ and $\frac{1}{\lambda}$ are distinct.
Similarly if $\psi \in S_\psi$ then $\frac{1}{\psi} \in S_\psi$ with $\psi$ and $\frac{1}{\psi}$ distinct.
We define the equivalence relation $\sim$ on $\mathbb{F}_p\setminus \{0, 1\}$ by
$$x \sim y \iff y \equiv x \textup{ or } \tfrac{1}{x} \pmod p.$$
We can then consider the bijective map $g: S_\lambda /{\sim} \to  S_\psi /{\sim}$ given by
$g([\lambda]) = \left[ \bigl(\tfrac{1-t}{1+t}\bigr)^2 \right]$,
with inverse map given by
$g^{-1}([\psi]) = \left[ \bigl(\tfrac{1-s}{1+s}\bigr)^2 \right]$,
where $s \in \mathbb{F}_p$ is fixed such that $s^2= \psi$. It is easy to check that these maps are well-defined and independent of choice of square roots, $t$ and~$s$.
There are exactly two elements in each equivalence class of $S_\lambda/{\sim}$, namely $\lambda$ and $\frac{1}{\lambda}$.
Also by Proposition \ref{prop_Etwist}(1), (\ref{for_aTwist}) and the fact that $\phi(\lambda)=1$  we see that  $a_p(\lambda)=a_p\bigl(\tfrac{1}{\lambda}\bigr)$.
Similarly there are exactly two elements in each equivalence class of $S_\psi /{\sim}$ with $a_p(\cdot)$ the same for both. Thus
\[
\sum_{\substack{\lambda =2\\ \chi_4(\lambda) = 1 \\ \phi(\lambda-1)=-1}}^{p-1} a_p(\lambda)
= 2 \sum_{[\lambda] \in S_\lambda /{\sim}} a_p(\lambda)
= 2 \sum_{[\psi] \in S_\psi /{\sim}} a_p(\psi)
= \sum_{\substack{\psi =2\\ \phi(\psi)=1, \chi_4(\psi)=-1 \\ \phi(\psi-1) = 1}}^{p-1} a_p(\psi).
\]
\end{proof}


\section{Proofs}\label{sec_proofs}

\begin{lemma} \label{lem_Fourier}
The space $S^{\new}_3(\Gamma_0(32), (\tfrac{-4}{\cdot}))$ is $2$-dimensional over $\mathbb{C}$ with basis $\{ f_2, \tilde{f}_2\}$, where $f_2 = \sum_{n=1}^\infty b(n)q^n = q + 4iq^3 +2 q^5 -8iq^7 + \cdots$, $i = \sqrt{-1}$, and $\tilde{f}_2$ is obtained by applying complex conjugation to the Fourier coefficients of $f_2$.
Furthermore, for an odd prime $p$, $b(p)$ is real if $p \equiv 1 \pmod{4}$ and purely imaginary if $p \equiv 3 \pmod{4}$.
\end{lemma}

\begin{proof}
Using Sage Mathematics Software~\cite{Sage}, one computes easily the dimension of the space as well as the first several Fourier coefficients of $f_2$ and $\tilde{f}_2$.
To see the second part of the lemma, we consider the twist $f_{2,\psi} = \sum_{n=1}^\infty \psi(n)b(n)q^n$ of $f_2$ by the character $\psi(\cdot) = (\tfrac{-4}{\cdot})$, which by \cite[Lem.~3.6]{Sh73} is a form in $S_3(\Gamma_0(32),\psi)$.  Likewise, the twist $\tilde{f}_{2,\psi}$ of $\tilde{f}_2$ is also in $S_3(\Gamma_0(32),\psi)$.  One checks via Sage that $f_{2,\psi}$ and $\tilde{f}_{2,\psi}$ are both new at level $32$.  Therefore, we can express the twists in terms of our basis for $S_3^{\new}(\Gamma_0(32),\psi)$, and it must be that
\[
  f_{2,\psi} = \tilde{f}_2, \quad \tilde{f}_{2,\psi} = f_2.
\]
We then have
$f_2 + \tilde{f}_2 = \sum_{n=1}^\infty 2\,\textup{Re}(b(n)) q^n = \sum_{n \equiv 1 \jmod{4}} 2b(n)q^n$, providing the result.
\end{proof}

\begin{proof}[Proof of Theorem \ref{thm_3F2_32}]
Assume that $p \equiv 1 \pmod{4}$ throughout, and choose a character $\chi_4 \in \widehat{\mathbb{F}_p^*}$ of order $4$.  By Lemma~\ref{lem_Fourier},
\[
b(p)=\tfrac{1}{2} \textup{Tr}^{\new}_3(\Gamma_0(32), (\tfrac{-4}{\cdot}),p),
\]
and so by Corollary \ref{cor_Trace32New}, it suffices to prove
\begin{equation}\label{for_rtp32}
{_{3}F_2} \biggl( \begin{array}{ccc} \chi_4, & \phi, & \phi \vspace{.02in}\\
\phantom{\chi_4} & \varepsilon, & \varepsilon \end{array}
\Big| \; 1 \biggr)_{p}
=
2-\displaystyle\sum_{\substack{0< |s| <2 \sqrt{p} \\ s  \equiv p+1 \imod {16}}} s \; \displaystyle\sum_{\substack{f \mid t\\ s^2-4p=t^2D}} h^*\left(\tfrac{s^2-4p}{f^2}\right) c_3(s,f)/2,
\end{equation}
where $D$ is a fundamental discriminant and $c_3(s,f)$ is as described in the Corollary \ref{cor_Trace32New}.
By Greene~\cite[Thm.~3.13]{G} (with \cite[Prop.~2.5]{McC6}) we have
\begin{equation*}
{_{3}F_2} \biggl( \begin{array}{ccc} \chi_4, & \phi, & \phi \vspace{.02in}\\
\phantom{\chi_4} & \varepsilon, & \varepsilon \end{array}
\Big| \; 1 \biggr)_{p}
=
- \chi_4(-1)
\sum_{\lambda=2}^{p-1}
{_{2}F_1} \biggl( \begin{array}{cc}\phi, & \phi \vspace{.02in}\\
\phantom{\phi} & \varepsilon \end{array}
\Big| \; \lambda \biggr)_{p} \;
\chi_4(\lambda) \bar{\chi_4}(1-\lambda).
\end{equation*}
Using a result of Koike \cite[Sec.~4]{Kk} (again with \cite[Prop.~2.5]{McC6}), we express the ${_{2}F_1}$ in the sum above in terms of the quantity $a_p(\lambda)=a_p(E_{\lambda})$, which we defined in Section \ref {sec_ec}.
This yields
\begin{equation}\label{for_3F2_to_ap32}
{_{3}F_2} \biggl( \begin{array}{ccc} \chi_4, & \phi, & \phi \vspace{.02in}\\
\phantom{\phi} & \varepsilon, & \varepsilon \end{array}
\Big| \; 1 \biggr)_{p}
= - \sum_{\lambda=2}^{p-1}
a_p(\lambda) \,
\chi_4(\lambda(\lambda-1)) \, \phi(\lambda-1).
\end{equation}
Using Lemmas \ref{lem_ap3} to \ref{lem_ap6} we see that
\begin{align}\label{for_apSplit32}
\notag - \sum_{\lambda=2}^{p-1}  a_p(\lambda) \, \chi_4(\lambda(\lambda-1)) \, \phi(\lambda-1)
&= -\sum_{\substack{\lambda =2\\ \phi(\lambda)=1}}^{p-1} a_p(\lambda) \chi_4(\lambda) \phi(\lambda-1)\\
\notag &= - \sum_{\substack{\lambda =2\\ \phi(\lambda)=1}}^{p-1} a_p(\lambda)
+\sum_{\substack{\lambda =2\\ \phi(\lambda)=1}}^{p-1} a_p(\lambda) [1-\chi_4(\lambda) \phi(\lambda-1)]\\
\notag &= 2+ 2 \sum_{\substack{\lambda =2\\ \phi(\lambda)=1\\ \chi_4(\lambda) \phi(\lambda-1)=-1}}^{p-1} a_p(\lambda) \\
&=2+4 \sum_{\substack{\lambda =2\\ \chi_4(\lambda)=1\\ \phi(\lambda-1)=-1}}^{p-1} a_p(\lambda).
\end{align}
We now use Theorem \ref{thm_Schoof} to show that
\begin{equation}\label{for_ap_to_HStar32}
\sum_{\substack{\lambda =2\\ \chi_4(\lambda)=1\\ \phi(\lambda-1)=-1}}^{p-1} a_p(\lambda)
=2 \sum_{\substack{0 < |s| < 2\sqrt{p}\\s \equiv p+1 \imod{16}}} s \,
\left(H^{\ast} \left(\tfrac{s^2-4p}{4}\right) - H^{\ast} \left(\tfrac{s^2-4p}{16}\right) \right).
\end{equation}
Define the set
$$L(s,p):= \{ \lambda \mid 2 \leq \lambda \leq p-1, \chi_4(\lambda)=1, \phi(\lambda-1)=-1, a_p(\lambda)=s\}.$$
Then, by the Hasse bound and Lemma \ref{lem_8tors}(1), we obtain
\begin{equation}\label{for_ap_to_L32}
 \sum_{\substack{\lambda=2 \\ \chi_4(\lambda)=1 \\ \phi(\lambda-1)=-1}}^{p-1} a_p(\lambda)
= \sum_{0 < |s| < 2\sqrt{p}} \; \sum_{\lambda \in L(s,p)} s
= \sum_{\substack{0 < |s| < 2\sqrt{p}\\s \equiv p+1 \imod{16}}} s \; |L(s,p)|.
\end{equation}
Let $\mathcal{I}_p$ denote the set of all isomorphism classes of elliptic curves over $\mathbb{F}_p$. Define
\[
I(s,p):= \left\{\mathcal{C} \in \mathcal{I}_p \biggm|
\parbox{2.75in}{\begin{center}
\textnormal{$\forall E \in \mathcal{C}$, $a_p(E) = s$, $E(\mathbb{F}_p)$ contains} \\
\textnormal{$\mathbb{Z} / 2 \mathbb{Z} \times  \mathbb{Z} /  8 \mathbb{Z}$  but not $\mathbb{Z} / 4 \mathbb{Z} \times  \mathbb{Z} /  4\mathbb{Z}$}
\end{center}}
\right\}.
\]
We now consider the map $F:L(s,p) \to I(s,p)$ given by $\lambda \mapsto [E_\lambda]$, where $[E_\lambda] \in \mathcal{I}_p$ is the isomorphism class containing $E_\lambda$.
Lemma \ref{lem_8tors} ensures $F$ is well-defined and surjective.
For $\lambda \in L(s,p)$ we see by Proposition \ref{prop_Etwist} and (\ref{for_aTwist}) that $E_{\lambda}$ is isomorphic to $E_{\frac{1}{\lambda}}$, $E_{1-\lambda}$ and $E_{ \frac{\lambda-1}{\lambda}}$ but not to $E_{ \frac{\lambda}{\lambda-1}}$ nor $E_{ \frac{1}{1-\lambda}}$ .
However, if $\lambda \in L(s,p)$ then $\tfrac{1}{\lambda}  \in L(s,p)$ but $1-\lambda$, $\tfrac{\lambda-1}{\lambda} \not\in L(s,p)$, as they do not meet the conditions.

By Proposition \ref{prop_Legendrej} we then see that, for a given $\lambda \in L(s,p)$, if $j([E(\lambda)]) \neq 0 ,1728$ then $F$ is two-to-one. We note that there are no classes of curves in $I(s,p)$ with $j=0$ or $1728$.
Otherwise, in the case $j=0$, $L(s,p)$ would have to contain $\lambda =\frac{1 \pm \sqrt{-3}}{2}$ which satisfies $\lambda^2-\lambda+1 = 0$, i.e, $\lambda-1$ is a square. In the case $j=1728$, $L(s,p)$ would have to contain $\lambda \in \{2, -1, \frac{1}{2}\}$, none of which satisfy the conditions.
Also, as $s\neq0$ we note that there are no classes of supersingular curves in  $I(s,p)$.
Hence, $\mathcal{O}\left(\tfrac{s^2-4p}{4}\right) \subset \mathcal{O}\left(\tfrac{s^2-4p}{16}\right) \not\subseteq \mathbb{Z}\left[\tfrac{-1+\sqrt{-3}}{2}\right]$ or $ \mathbb{Z}\left[\sqrt{-1}\right]$.
(See \cite[\S 3]{Sc} for details.)

By Theorem \ref{thm_Schoof} we know that $|I(s,p)| = \left( H\left(\tfrac{s^2-4p}{4}\right) - H\left(\tfrac{s^2-4p}{16}\right) \right)$, as $s \equiv p+1 \pmod{16}$. Therefore
\begin{equation*}
|L(s,p)|=2 \left( H\left(\tfrac{s^2-4p}{4}\right) - H\left(\tfrac{s^2-4p}{16}\right) \right) =2 \left(H^{\ast} \left(\tfrac{s^2-4p}{4}\right) - H^{\ast} \left(\tfrac{s^2-4p}{16}\right) \right),
\end{equation*}
and \eqref{for_ap_to_HStar32} holds via \eqref{for_ap_to_L32}.
So, combining \eqref{for_3F2_to_ap32}, \eqref{for_apSplit32} and \eqref{for_ap_to_HStar32} we obtain
\begin{equation}\label{for_3F2_to_HStar32}
{_{3}F_2} \biggl( \begin{array}{ccc} \chi_4, & \phi, & \phi \vspace{.02in}\\
\phantom{\phi} & \varepsilon, & \varepsilon \end{array}
\Big| \; 1 \biggr)_{p}
= 2 + 8 \sum_{\substack{0 < |s| < 2\sqrt{p}\\s \equiv p+1 \imod{16}}} s \, \left(H^{\ast} \left(\tfrac{s^2-4p}{4}\right) - H^{\ast} \left(\tfrac{s^2-4p}{16}\right) \right).
\end{equation}
Comparing (\ref{for_rtp32}) and (\ref{for_3F2_to_HStar32}) it now suffices to show
\begin{multline}\label{for_H_to_h32}
8 \sum_{\substack{0 < |s| < 2\sqrt{p}\\s \equiv p+1 \imod{16}}} s \, \left(H^{\ast} \left(\tfrac{s^2-4p}{4}\right) - H^{\ast} \left(\tfrac{s^2-4p}{16}\right) \right)\\
=
-\displaystyle\sum_{\substack{0< |s| <2 \sqrt{p} \\ s  \equiv p+1 \imod {16}}} s \; \displaystyle\sum_{\substack{f \mid t\\ s^2-4p=t^2D}} h^*\left(\tfrac{s^2-4p}{f^2}\right) c_3(s,f)/2,
\end{multline}
where $D$ is a fundamental discriminant and $c_3(s,f)$ is as described in the Corollary \ref{cor_Trace32New}.

We now use (\ref{def_BigH}) and Lemma \ref{lem_hrel} to show that (\ref{for_H_to_h32}) holds. We proceed on a case by case basis depending on the congruence class of both $p$ and $D$.
Recall $a:=\ord_2(t)$ and $b:=\ord_2(f)$. Let $p \equiv 1 \pmod 8$ and $D \equiv 1 \pmod 8$. Then by Proposition \ref{prop_pDa} we have $a>2$.
We first examine the left-hand side of (\ref{for_H_to_h32}). By Lemma \ref{lem_hrel}
$$h^*\left(\tfrac{s^2-4p}{f^2}\right) = h^*\left(\tfrac{s^2-4p}{(2f)^2}\right) \times
\begin{cases}
2 & \textup{if } a-b \geq 2,\\
1 & \textup{if } a-b = 1,
\end{cases}$$
and
$$h^*\left(\tfrac{s^2-4p}{f^2}\right) = h^*\left(\tfrac{s^2-4p}{(4f)^2}\right) \times
\begin{cases}
4 & \textup{if } a-b \geq 3,\\
2 & \textup{if } a-b = 2.
\end{cases}$$
Therefore using (\ref{def_BigH}) we see that
\begin{align}\label{for_LHS32_1}
\notag 8 \, \left(H^{\ast} \left(\tfrac{s^2-4p}{4}\right) - H^{\ast} \left(\tfrac{s^2-4p}{16}\right) \right)
&= 8 \displaystyle\sum_{\substack{f \mid t\\ a-b \geq 1}} h^*\left(\tfrac{s^2-4p}{(2f)^2}\right)
- 8 \displaystyle\sum_{\substack{f \mid t\\ a-b \geq 2}} h^*\left(\tfrac{s^2-4p}{(4f)^2}\right)\\
&=2  \displaystyle\sum_{\substack{f \mid t\\ a-b \geq 3}} h^*\left(\tfrac{s^2-4p}{f^2}\right)
+ 8 \displaystyle\sum_{\substack{f \mid t\\ a-b = 1}} h^*\left(\tfrac{s^2-4p}{f^2}\right).
\end{align}
Now the coefficient of $s$ in the expression on the right-hand side of (\ref{for_H_to_h32}) in this case equals
\begin{equation}\label{for_RHS32_1}
2 \displaystyle\sum_{\substack{f \mid t\\ a-b \geq 3}} h^*\left(\tfrac{s^2-4p}{f^2}\right) +
2 \displaystyle\sum_{\substack{f \mid t\\ a-b = 2}} h^*\left(\tfrac{s^2-4p}{f^2}\right) +
3 \displaystyle\sum_{\substack{f \mid t\\ a-b = 1}} h^*\left(\tfrac{s^2-4p}{f^2}\right) +
1 \displaystyle\sum_{\substack{f \mid t\\ a-b = 0}} h^*\left(\tfrac{s^2-4p}{f^2}\right).
\end{equation}
As $a>2$,
\begin{align*}
\displaystyle\sum_{\substack{f \mid t\\ a-b = 0}} h^*\left(\tfrac{s^2-4p}{f^2}\right)
=
\displaystyle\sum_{\substack{f \mid t\\ a-b = 1}} h^*\left(\tfrac{s^2-4p}{(2f)^2}\right)
=  \displaystyle\sum_{\substack{f \mid t\\ a-b = 1}} h^*\left(\tfrac{s^2-4p}{f^2}\right),
\end{align*}
and
\begin{align*}
2 \displaystyle\sum_{\substack{f \mid t\\ a-b = 2}} h^*\left(\tfrac{s^2-4p}{f^2}\right)
=
4 \displaystyle\sum_{\substack{f \mid t\\ a-b = 2}} h^*\left(\tfrac{s^2-4p}{(2f)^2}\right)
= 4 \displaystyle\sum_{\substack{f \mid t\\ a-b = 1}} h^*\left(\tfrac{s^2-4p}{f^2}\right).
\end{align*}
Thus (\ref{for_RHS32_1}) equals
\begin{equation}\label{for_RHS32_2}
2 \displaystyle\sum_{\substack{f \mid t\\ a-b \geq 3}} h^*\left(\tfrac{s^2-4p}{f^2}\right) +
8 \displaystyle\sum_{\substack{f \mid t\\ a-b = 1}} h^*\left(\tfrac{s^2-4p}{f^2}\right).
\end{equation}
Expressions (\ref{for_RHS32_2}) and (\ref{for_LHS32_1}) are equal and so (\ref{for_H_to_h32}) holds in the case $p \equiv 1 \pmod 8$ and $D \equiv 1 \pmod 8$. The other cases proceed in a similar manner. We omit the details for reasons of brevity.
We conclude, therefore, that Theorem \ref{thm_3F2_32} is true.
\end{proof}


\begin{remark}
Using the same methods as in the proof of Theorem~\ref{thm_3F2_32}, together with Theorem~\ref{thm_Trace16}, we obtain a new proof of Theorem~\ref{thm_3F2_16}, but we do not present the details here.
\end{remark}


We now recall a theorem of the first author that is a finite field version of Whipple's ${}_4F_3$-hypergeometric transformation for well-poised series.  If $A \in \widehat{\mathbb{F}_p^{*}}$ is a square we will write $A = \square$.

\begin{theorem}[{McCarthy~\cite[Thm.~1.5]{McC6}}]\label{thm_4F3_3F2}
For $A$, $B$, $C$, $D \in \widehat{\mathbb{F}_p^{*}}$,
\begin{multline*}
{_{4}F_3} \biggl( \begin{array}{cccc} A, & B, & C, & D \vspace{.02in}\\
\phantom{A} & A\bar{B}, & A\bar{C}, & A\bar{D}  \end{array}
\Big| \; {-1} \biggr)_{p}\\
=
\begin{cases}
0 & \textup{if $A \neq \square$,}\\[4pt]
\dfrac{g(\bar{A}) \, g(\bar{A}CD) }{ g(\bar{A}C) \, g( \bar{A}D)}
\displaystyle\sum_{R^2=A}
{_{3}F_2} \biggl( \begin{array}{ccc} R\bar{B}, & C, & D \vspace{.02in}\\
\phantom{A} & R, & A\bar{B} \end{array}
\Big| \; 1 \biggr)_{p}
& \genfrac{}{}{0pt}{0} {\textup{if $A = \square$,  $A \neq \varepsilon$, $B \neq \varepsilon$,}}{\;\;\; \textup{$B^2 \neq A$ and $CD \neq A$.}}
\end{cases}
\end{multline*}
\end{theorem}


\begin{proof}[Proof of Theorem \ref{thm_4F3}]
Van Geemen and van Straten \cite[\S 8.7]{vGvS} conjectured that there exists a relationship between the Andrianov $L$-function of the form $F_7$ and the tensor product $L$-function of the forms $f_1$ and $f_2$.
As mentioned previously we understand that this conjecture has been resolved by Okazaki \cite{Ok} and is the subject of a forthcoming paper (see also ~\cite[\S 1]{Ok2} for additional discussion on this problem).
Specifically, the conjecture implies that there exists an eigenform $F \in S_3^{2}(\Gamma^2(4,8))$ such that $F = F_7 | \gamma$ for some $\gamma \in \textup{Sp}_4(\mathbb{Z})$ and
$$ L^{a}(s,F) = L(s, f_1 \otimes f_2).$$
The Andrianov $L$-Function for forms of degree 2 with level is described in \cite{Ev}.
We see that $L^{a}(s,F_7)$ has an Euler product over odd primes with local factor
\begin{equation}\label{for_La_pfactorF7}
L^{a}(s,F_7)_p = \left[ 1 - \lambda(p) p^{-s} + (\lambda(p)^2- \lambda(p^2) -  p^{2}) p^{-2s} -  \lambda(p) p^{3-3s} + p^{6-4s} \right]^{-1},
\end{equation}
where $\lambda(p^2)$ is the eigenvalue associated to the Hecke operator of index $p^2$.  From \cite[Prop.~2.2]{Ok2} we then obtain
\[
L^{a}(s,F)_p = \left[ 1 - \zeta(p) \lambda(p) p^{-s} + (\lambda(p)^2- \lambda(p^2) -  p^{2}) p^{-2s} -  \zeta(p) \lambda(p) p^{3-3s} + p^{6-4s} \right]^{-1}
\]
for a certain function $\zeta(t)$ which is defined on odd $t$ modulo $8$ (see also \cite[Thm.~2.2]{AA}).

We now examine $L(s, f_1 \otimes f_2)$. Let $\psi(\cdot) =  (\tfrac{-4}{\cdot})$ be the character associated to $f_2$.  Given that $f_1$ and $f_2$ are both Hecke eigenforms we know by the work of Shimura \cite{Sh} that
$L(s, f_1 \otimes f_2)$ has an Euler product with local factors at odd primes of
\begin{multline*}
L(s, f_1 \otimes f_2)_p = \left[1 - a(p) b(p) p^{-s} + (p \, b(p)^2 + \psi(p) \, p^2 \, a(p)^2 - 2 \, \psi(p) \, p^3) p^{-2s}
\right. \\ \left.
- \psi(p) a(p) b(p) p^{3-3s} +p^{6-4s} \right]^{-1} .
\end{multline*}
We note also that $L(s, f_1 \otimes f_2)_2 = \left[1 - a(2) b(2) 2^{-s}\right]^{-1} = 1.$

Comparing the coefficients of $p^{-s}$ and $p^{-3s}$ in $L^a(s,F)_p$ and $L(s,f_1 \otimes f_2)_p$ we conclude that
$$ a(p) b(p) = \zeta(p) \lambda(p) = \bigl( \tfrac{-4}{p} \bigr) a(p) b(p) .$$
Therefore, we conclude that $\lambda(p) = 0$ when $p \equiv 3\pmod{4}$.
This is also apparent from the fact that $f_1$ is a CM-form and $a(p)=0$ for primes $p \equiv 3 \pmod 4$ (e.g., see \cite[Prop.~1]{O}).
For $p\equiv 1 \pmod 4$ we determine $\zeta(p)$ by examining the values of $a(p)$, $b(p)$ and $\lambda(p)$ when $p=5$, $17$, which are listed in \cite[\S\S 7--8]{vGvS}.
We determine therefore that $\zeta(p)= 1$, if $p \equiv 1 \pmod{8}$, and $\zeta(p)=-1$, if $p\equiv 5 \pmod{8}$, which coincides with $\xi(p)$.
Overall then, we find
\begin{equation*}
 \lambda(p) =
\begin{cases}
{\xi}(p) \, a(p) \, b(p) & \textup{if } p\equiv 1 \pmod 4,\\
0 & \textup{if }p\equiv 3 \pmod 4.
\end{cases}
\end{equation*}
By Theorem \ref{thm_4F3_3F2} we have that, when $p\equiv 3 \pmod 4$,
$${_{4}F_3} \biggl( \begin{array}{cccc} \phi, & \phi, & \phi, & \phi \vspace{.02in}\\
\phantom{\phi} & \varepsilon, & \varepsilon, & \varepsilon \end{array}
\Big| \; {-1} \biggr)_{p}=0.$$
Thus the theorem is proved in the case $p\equiv 3 \pmod 4$.

We now concentrate on the case when $p\equiv 1 \pmod 4$ and assume this throughout the remainder of the proof.  Combining Theorems~\ref{thm_2F1_32} and~\ref{thm_3F2_32} it suffices to show that
\begin{equation}\label{for_4F3_2F1x3F2}
{_{4}F_3} \biggl( \begin{array}{cccc} \phi, & \phi, & \phi, & \phi \vspace{.02in}\\
\phantom{\phi} & \varepsilon, & \varepsilon, & \varepsilon \end{array}
\Big| \; {-1} \biggr)_{p}
= {_{2}F_1} \biggl( \begin{array}{cc} \phi, & \phi \vspace{.02in}\\
\phantom{\phi} & \varepsilon,  \end{array}
\Big| \; {-1} \biggr)_{p}
 \cdot
{_{3}F_2} \biggl( \begin{array}{ccc} \chi_4, & \phi, & \phi \vspace{.02in}\\
\phantom{\chi_4} & \varepsilon, & \varepsilon \end{array}
\Big| \; 1 \biggr)_{p}.
\end{equation}
Now Theorem \ref{thm_4F3_3F2} yields
\[
{_{4}F_3} \biggl( \begin{array}{cccc} \phi, & \phi, & \phi, & \phi \vspace{.02in}\\
\phantom{\phi} & \varepsilon, & \varepsilon, & \varepsilon \end{array}
\Big| \; {-1} \biggr)_{p}
=p \left[
{_{3}F_2} \biggl( \begin{array}{ccc} \bar{\chi_4}, & \phi, & \phi \vspace{.02in}\\
\phantom{\bar{\chi_4}} & \chi_4, & \varepsilon \end{array}
\Big| \; 1 \biggr)_{p}
+
{_{3}F_2} \biggl( \begin{array}{ccc} \chi_4, & \phi, & \phi \vspace{.02in}\\
\phantom{\chi_4} & \bar{\chi_4}, & \varepsilon \end{array}
\Big| \; 1 \biggr)_{p}
\right].
\]
We use transformations of Greene \cite[(4.23), (4.25)]{G} (with \cite[Prop.~2.5]{McC6}) to relate both $_3F_2$'s in this equation to the $_3F_2$ in (\ref{for_4F3_2F1x3F2}).  This gives us
\begin{equation*}\label{for_4F3_3F2_1}
{_{4}F_3} \biggl( \begin{array}{cccc} \phi, & \phi, & \phi, & \phi \vspace{.02in}\\
\phantom{\phi} & \varepsilon, & \varepsilon, & \varepsilon \end{array}
\Big| \; {-1} \biggr)_{p}
= - {_{3}F_2} \biggl( \begin{array}{ccc} \chi_4, & \phi, & \phi \vspace{.02in}\\
\phantom{\chi_4} & \varepsilon, & \varepsilon \end{array}
\Big| \; 1 \biggr)_{p} \times
\left[\frac{g(\phi) \, g(\chi_4)}{g(\bar{\chi_4})} + \frac{g(\phi) \, g(\bar{\chi_4})}{g(\chi_4)} \right].
\end{equation*}
By \cite[Thm.~1.10]{McC6},
\begin{equation*}
 {_{2}F_1} \biggl( \begin{array}{cc} \phi, & \phi \vspace{.02in}\\
\phantom{\phi} & \varepsilon,  \end{array}
\Big| \; {-1} \biggr)_{p}
= -\frac{g(\phi) \, g(\chi_4)}{g(\bar{\chi_4})} - \frac{g(\phi) \, g(\bar{\chi_4})}{g(\chi_4)}.
\end{equation*}
Therefore (\ref{for_4F3_2F1x3F2}) holds and the theorem is proved.
\end{proof}


\section{Closing Remarks}\label{sec_remarks}
We first note that the definition of $_{n+1}F_n $ as described in  (\ref{def_HypFnFF}) can be easily extended to finite fields with a prime power number of elements \cite{McC6}.  We also note that the representation of the Hecke algebra in $M_k^2(\Gamma^{2}(q))$ is generated by $T(p)$ and $T(p^2)$, the Hecke operators of index $p$ and $p^2$ respectively \cite{Ev}.
Therefore, for a given eigenform $F$ of degree 2 we are also interested in $\lambda(p^2)$, the eigenvalue associated to the action of $T(p^2)$ on $F$.
Based on (limited) numerical evidence it appears that the eigenvalue $\lambda(p^2)$ associated to the eigenform $F_7$ is also related to a hypergeometric function but over $\mathbb{F}_{p^2}$, as follows.
Let $a_p:=\lambda(p)$ and $a_{p^2}:=\lambda(p)^2- \lambda(p^2) -  p^{2}$ corresponding (up to sign) to the coefficients of $p^{-s}$ and $p^{-2s}$ in (\ref{for_La_pfactorF7}) respectively.
Then we have observed for primes $p<20$ that
$$a_p^2 - 2 \, a_{p^2}
= {_{4}F_3} \biggl( \begin{array}{cccc} \phi, & \phi, & \phi, & \phi \vspace{.02in}\\
\phantom{\phi} & \varepsilon, & \varepsilon, & \varepsilon \end{array}
\Big| \; {-1} \biggr)_{p^2}.$$
It is conceivable that this identity holds for all odd primes $p$, and we believe similar methods to those employed in the main body of this paper could be applied to proving it.


\end{document}